\documentclass[12pt]{amsart}
\usepackage{amscd,amsmath,amsthm,amssymb,color}

\usepackage{pstcol,pst-plot,pst-3d}
\usepackage{lmodern,pst-node}
\usepackage{multicol}
\usepackage{epic,eepic}
\usepackage{amsfonts,amssymb,amscd,amsmath,enumerate,verbatim}
\psset{unit=0.7cm,linewidth=0.8pt,arrowsize=2.5pt 4}

\newpsstyle{fatline}{linewidth=1.5pt}
\newpsstyle{fyp}{fillstyle=solid,fillcolor=verylight}
\definecolor{verylight}{gray}{0.97}
\definecolor{light}{gray}{0.9}
\definecolor{medium}{gray}{0.85}
\definecolor{dark}{gray}{0.6}
\def\frk{\frak}               

\def\Phi{{\frk n}}
\def\Phi{{\frk N}}
\def\opn#1#2{\def#1{\operatorname{#2}}} 
\opn\chara{char} \opn\length{\ell} \opn\pd{pd} \opn\rk{rk}
\opn\projdim{proj\,dim} \opn\injdim{inj\,dim} \opn\rank{rank}
\opn\depth{depth} \opn\grade{grade} \opn\height{height}
\opn\embdim{emb\,dim} \opn\codim{codim}

\opn\Tr{Tr} \opn\bigrank{big\,rank}
\opn\superheight{superheight}\opn\lcm{lcm}
\opn\trdeg{tr\,deg}
\opn\reg{reg} \opn\lreg{lreg} \opn\ini{in} \opn\lpd{lpd}
\opn\size{size} \opn\sdepth{sdepth}
\opn\link{link}\opn\fdepth{fdepth}\opn\lex{lex}
\opn\div{div} \opn\Div{Div} \opn\cl{cl} \opn\Cl{Cl}
\opn\Spec{Spec} \opn\Supp{Supp} \opn\supp{supp} \opn\Sing{Sing}
\opn\Ass{Ass} \opn\Min{Min}\opn\Mon{Mon}
\opn\Ann{Ann} \opn\Rad{Rad} \opn\Soc{Soc}
\opn\Im{Im} \opn\Ker{Ker} \opn\Coker{Coker} \opn\Am{Am}
\opn\Hom{Hom} \opn\Tor{Tor} \opn\Ext{Ext} \opn\End{End}
\opn\Aut{Aut} \opn\id{id}

\opn\nat{nat}
\opn\pff{pf}
\opn\Pf{Pf} \opn\GL{GL} \opn\SL{SL} \opn\mod{mod} \opn\ord{ord}
\opn\Gin{Gin} \opn\Hilb{Hilb}\opn\sort{sort}
\opn\aff{aff} \opn\con{conv} \opn\relint{relint} \opn\st{st}
\opn\lk{lk} \opn\cn{cn} \opn\core{core} \opn\vol{vol}
\opn\link{link} \opn\star{star}\opn\lex{lex}\opn\set{set}\opn\rev{rev}
\opn\gr{gr}

\def\pot#1#2{#1[\kern-0.28ex[#2]\kern-0.28ex]}

\opn\dirlim{\underrightarrow{\lim}}
\opn\inivlim{\underleftarrow{\lim}}
\def\Implies{\ifmmode\Longrightarrow \else
        \unskip${}\Longrightarrow{}$\ignorespaces\fi}
\def\implies{\ifmmode\Rightarrow \else
        \unskip${}\Rightarrow{}$\ignorespaces\fi}
\def\iff{\ifmmode\Longleftrightarrow \else
        \unskip${}\Longleftrightarrow{}$\ignorespaces\fi}

\let\:=\colon
\newtheorem{Theorem}{Theorem}[section]

\newtheorem{Corollary}[Theorem]{Corollary}
\newtheorem{Proposition}[Theorem]{Proposition}

\newtheorem{Remarks}[Theorem]{Remarks}

\let\epsilon\varepsilon
\let\kappa=\varkappa
\def\qed{\ifhmode\textqed\fi
      \ifmmode\ifinner\quad\qedsymbol\else\dispqed\fi\fi}
\def\textqed{\unskip\nobreak\penalty50
       \hskip2em\hbox{}\nobreak\hfil\qedsymbol
       \parfillskip=0pt \finalhyphendemerits=0}
\def\dispqed{\rlap{\qquad\qedsymbol}}

\opn\dis{dis}
\def\pnt{{\raise0.5mm\hbox{\large\bf.}}}

\opn\Lex{Lex}

\begin{document}

\title {Ideals generated by $2$--minors of Hankel matrices}

\author {Faryal Chaudhry and Ayesha Asloob Qureshi }

\address{Faryal Chaudhry, Abdus Salam School of Mathematical Sciences, GC University,
Lahore. 68-B, New Muslim Town, Lahore 54600, Pakistan} \email{chaudhryfaryal@gmail.com}

\address{Ayesha Asloob Qureshi, Department of Pure and Applied Mathematics, Graduate School of Information Science and Technology,
Osaka University, Toyonaka, Osaka 560-0043, Japan}
\email{ayesqi@gmail.com}

\begin{abstract}
We study ideals generated by $2$--minors of generic Hankel matrices.
\end{abstract}

\thanks{The first author acknowledges the support from Higher Education Commission, Pakistan. The second author was supported by JSPS Postdoctoral Fellowship Program for Foreign Researchers}
\subjclass{13H10, 13P10, 13C40}
\keywords{Determinantal ideals, Hankel matrices, closed graphs}
\maketitle

\section*{Introduction}

In \cite{CDE}, the authors introduced and studied binomial edge ideals associated with scrolls. More precisely, to a closed graph $G$ on the vertex set $[n]$ with the edge set $E(G)$, one associates the binomial ideal $I_{G} \subset K[x_1,\ldots,x_{n+1}]$ generated by the $2$-minors of the $(2\times n)$ - Hankel matrix
$\begin{pmatrix}
x_{1}  & x_{2}   & \cdots     &  x_{n}      \\
x_{2}  & x_{3}   & \cdots     &  x_{n+1}
\end{pmatrix}$
which correspond to the edges of the graph $G$. In other words, $I_{G}=(\begin{vmatrix}
x_{i}  & x_{j}      \\
x_{i+1}  & x_{j+1}
\end{vmatrix}: i<j, \{i,j\} \in E(G))$.

The definition of scroll binomial edge ideals was inspired by the construction of the classical binomial edge ideals as they were introduced by \cite{HHHKR} and \cite{Oh}  a few years ago.
 Later on, there were considered several ways to generalize classical binomial edge ideals. We refer the reader to \cite{EHHQ,EHHM,R} for further information on these gene\-ralizations. Similar developments may be considered for scroll binomial edge ideals. One direction of generalization  is illustrated in this paper.

Namely, for a generic Hankel matrix $X = (x_{ij})_{\substack{1\leq i \leq m, \\ 1\leq j \leq n}}$ with $x_{ij}= x_{i+j-1}$ for all $i,j,$ and for two closed graphs $G_{1}$ on the vertex set $[m]$ with edge set $E(G_{1})$, and $G_{2}$ on the vertex set $[n]$ with edge set $E(G_{2})$, we consider the Hankel binomial ideal $I_{G_{1},G_{2}} \subset S = K[x_1,\ldots,x_{m+n-1}]$ defined as follows: $$ I_{G_{1},G_{2}}= ( g_{e,f}=\begin{vmatrix}
x_{i+k-1}  & x_{i+l-1}      \\
x_{j+k-1}  & x_{j+l-1}
\end{vmatrix}: e= \{i,j\} \in E(G_{1}), f= \{k,l\} \in E(G_{2})).$$
If $G_{1}$ and $G_{2}$ are complete graphs, that is, $G_{1}=K_{m}$ and $G_{2}=K_{n}$, then $I_{G_{1},G_{2}}$ is generated by all the $2$-minors of $X$. We refer the reader to \cite{C,W} for information about the ideal $I_{K_m,K_n}$.

In this paper, we work with closed graphs. We recall from \cite{HHHKR} that a simple graph $G$ on the vertex set $[n]$ is closed if there exists a labeling of its vertices with the property that  if $1 \leq i < j < k \leq n$ or $1 \leq k < j < i \leq n$, and $\{i,j\}$, $\{i,k\}$ are edges of $G$, then $\{j,k\}$ is an edge of $G$. This is equivalent to saying that if $\{i,j\}\in E(G)$ with $i<j,$ then, for all
$i<k<j,$ $\{i,k\}$ and $\{k,j\}$ are edges in $G.$

In \cite{EHH}, it was shown that a simple graph $G$ on $[n]$ is closed if and only if there exists a labeling of $G$ such that all facets of the clique complex $\Delta(G)$ of $G$ are intervals $[a,b] \subset [n]$.
A clique of $G$ is a complete subgraph of $G$. The cliques of $G$ form a simplicial complex called the clique complex of $G$.

Throughout this paper, if $G$ is a closed graph on $[n]$, we assume that $G$ is labeled such that if its maximal cliques are $F_1,\dots,F_r,$
then $ F_{i}=[a_{i},b_{i}]$ for $1\leq i\leq r$, and  $1=a_1<a_2<\cdots <a_r<b_r=n$; see \cite[Theorem 2.2]{EHH}.
In addition, we write $\Delta(G)=\langle F_1,\dots,F_r\rangle$ if the maximal cliques of $G$ are $F_1,\dots,F_r$.

Let $G_{1}$, $G_{2}$ be connected closed graphs on $[m]$, respectively $[n]$. To $G_{1}$ and $G_{2}$ we associate a graph $G$ on the vertex set $[m+n-2]$ with the edge set: $$E(G)=\{\{i+k-1, j+l-2\} : i<j, k<l,\{i,j\} \in E(G_{1}),~\text{and}~ \{k,l\} \in E(G_{2})\}.$$

In Theorem \ref{1.1}, we show that $I_{G_{1},G_{2}}= I_{G}$, where $I_{G}\subset S$ is the scroll binomial edge ideal of $G$. Moreover, $G$ is a connected closed graph. This is the main result of our paper. It allows us to apply all the known results on scroll binomial edge ideals proved in \cite{CDE}. The first consequence of Theorem~\ref{1.1} is that $I_{G_{1},G_{2}}$ has a quadratic Gr\"obner basis with respect to the revlexicographic order on $S$ induced by $x_1>\cdots >x_{m+n-1}$. Additionally, it follows that $I_{G_{1},G_{2}}$ is a Cohen- Macaulay ideal of dimension $2$.

In Proposition \ref{2.1}, we show that any maximal clique of the graph $G$ is actually obtained by "adding" a maximal clique $[a,b]$ of $G_{1}$ with a maximal clique of $G_{2}$. By using this proposition, in Theorem \ref{2.3} we derive the main properties of $I_{G_{1},G_{2}}$: primality, minimal primes, radical property, linear resolution.

Finally, in Proposition \ref{2.4}, we show that $\reg(S/I_{G_{1},G_{2}}) \leq {m+n-2}$ and the equality holds if and only if $G_{1}$ and $G_{2}$ are line graphs.

We would like to make a final remark. If one of the graphs $G_{1}$, $G_{2}$ is not connected and the other one is connected, then the associated graph $G$ is still connected, thus all the proved results are still valid. If both graphs are disconnected, then one easily sees that $G$ might be disconnected. In that case, we may apply only the results of \cite{CDE} which do not involve the connectedness of the graph G.
We chose to treat only the case when $G_{1}$ and $G_{2}$ are connected since this is the most interesting  setting   and to avoid long technical arguments needed for distinguish between those graphs $G_1$ and $G_2$ which give a connected or disconnected graph $G.$

\section{Gr\"obner basis}

Let $G_{1}$, $G_{2}$ be two connected closed graphs on the vertex $[m]$ and $[n]$, respectively, and $X$ be a generic $(m \times n)$ - Hankel matrix with $m \leq n$. Thus,
$$X=\left(
\begin{array}{cccc}
  x_{1} & x_{2} & \ldots & x_{n} \\
  x_{2} & x_{3} & \ldots & x_{n+1} \\
  \ldots & \ldots & \ldots & \ldots \\
  x_{m} & x_{m+1} & \ldots & x_{m+n-1}
\end{array}\right).
$$
Let $e=\{i,j\} \in E(G_{1})$ with $i<j$ and $f=\{k,l\} \in E(G_{2})$ with $k<l$. To the pair $(e,f)$, we assign the following $2$-minor of $X$: $$g_{e,f} =[i\ j|k\ l]= x_{i+k-1}x_{j+l-1}-x_{j+k-1}x_{i+l-1}.$$
We fix a field $K$ and let $S=K[x_{1},\ldots ,x_{m+n-1}]$ endowed with the reverse lexicographic order induced by $x_{1}>x_{2}> \cdots >x_{m+n-1}$. Then, with respect to this order, $\ini_{\rev}(g_{e,f})$ = $x_{j+k-1}x_{i+l-1}$.

Let $X'=\left(
\begin{array}{cccc}
  x_1 & x_2 & \ldots & x_{m+n-2} \\
  x_{2} & x_{3} & \ldots & x_{m+n-1}
\end{array}\right)
$
be the $2\times (m+n-2)$- Hankel matrix and $G$ be the graph on the vertex set $[m+n-2]$ whose edge set is:
$$E(G) = \{\{i+k-1, j+l-2\}: i<j, k<l, \{i,j\} \in E(G_{1})~ \text{and}~\{k,l\} \in E(G_{2})\}.$$

Let $G_{1}$ and $G_{2}$ be as before. We define the Hankel ideal of the matrix $X$ as
$$I_{G_{1},G_{2}} = (g_{e,f}: e \in E(G_{1}),~ f \in E(G_{2})).$$
In addition, let
$I_{G}=(g_{ij}= \begin{vmatrix}
x_{i}  & x_{j}      \\
x_{i+1}  & x_{j+1}
\end{vmatrix}: i<j, \{i,j\} \in E(G))$
be the scroll binomial edge ideal defined on the matrix $X'.$

With the above notation and settings we may state the main result of this paper.

\begin{Theorem}\label{1.1}
Let $G_1$ and $G_2$ be closed graphs. Then $G$ is a connected closed graph and   $I_{G_1,G_2} = I_G$.
\end{Theorem}

\begin{proof}

Let $\{p,q\}, \{p,r\} \in E(G)$ and $q<r$. Then $p=i+k-1$, $q=j+l-2$ and $r=u+v-2$ for some $\{i,j\}, \{i,u\} \in E(G_1)$ and $\{k,l\}, \{k,v\} \in E(G_2)$. We may assume that $j < u$ and $l<v$. Then the $\{j-1,u\} \in E(G_1)$ and $\{l,v\} \in E(G_2)$ because $G_1$ and $G_2$ are closed. This gives $\{j+l-2, v+u-2\} = \{q,r\} \in E(G)$. 

Similarly, if $\{p,q\}, \{r,q\} \in E(G)$ with $p<r<q$, then by similar arguments, it follows that $\{p,r\} \in E(G)$. Therefore, $G$ is a closed graph. For connectedness, it is enough to observe that,
for any $i\leq m-1$ and $k\leq n-1,$ $\{i,i+1\}\in E(G_1)$ and $\{k,k+1\}\in E(G_2),$ thus $\{i+k-1, i+k\}\in E(G).$

Next, we prove the equality $I_{G_1,G_2} = I_G$.
Let $e=\{i,j\} \in E(G_1)$ and $f=\{k,l\} \in E(G_2)$ and $e_f= \{i+k-1,j+l-2\} \in E(G)$. Then $h_{e_f}= x_{i+k-1} x_{j+l-1} - x_{i+k} x_{j+l-2}$ and $g_{e,f}= x_{i+k-1}x_{j+l-1} - x_{i+l-1} x_{j+k-1}$ are typical generators of $I_G$ and $I_{G_!,G_2}$, respectively. First, we  show that $I_G \subset I_{G_1,G_2}$. If $j=i+1$ or $l=k+1$ we get $h_{e_f}=g_{e,f},$, thus $h_{e_f}\in I_{G_1,G_2}.$ Now we consider
$j>i+1$ and $l>k+1.$
By using the fact that $G_1$ and $G_2$ are closed graphs, we see that $\{i,p\}, \{p,j\} \in E(G_1)$, and $\{k,q\}, \{q,l\} \in E(G_2)$ for all $i<p<j$ and $k<q<l$.  In particular, $e'=\{i+1,j\} \in E(G_1)$ and $f'=\{k,l-1\} \in E(G_2)$. Then $g_{e',f'} = x_{i+k}x_{j+l-2} - x_{i+l-1} x_{j+k-1} \in I_{G_1,G_2}$ and $h_{e_f}= g_{e,f}  - g_{e',f'} \in I_{G_1,G_2}$. Therefore, $I_G \subset I_{G_1,G_2}$.

Now, we show that $I_{G_1,G_2} \subset  I_G $.  Let $l-k > j-i=t $. Again, by using the fact the $G_1$ and $G_2$ are closed, we see that $e_1 =\{i+1, j\}, e_2 = \{i+2,j\}, \ldots, e_t= \{i+t-1,j\} \in E(G_1)$
and $f_1=\{k,l-1\}, f_2 = \{k,l-2\}, \ldots, f_t = \{k,l-t+1\} \in E(G_2)$. Then $g_{e,f} = h_{e_f} + h_{{e_1}_{f_1}}+  h_{{e_2}_{f_2}} + \cdots +  h_{{e_t}_{f_t}}$, which gives $g_{e,f} \in I_{G}$. Similarly, one can show $g_{e,f} \in I_{G}$ when $j-i > l-k$. This completes the proof.
\end{proof}

By applying \cite[Theorem 1.1]{CDE} and \cite[Corollary 1.3]{CDE}  we get the following consequence of the above theorem.

\begin{Corollary} Let $G_1,G_2$ be two connected closed graphs on the vertex sets $[m],$ respectively $[n].$ Then $I_{G_1,G_2}$ has a quadratic Gr\"obner basis with respect to the revlexicographic order induced by $x_1>\cdots >x_{m+n-1}$. Moreover, $I_{G_1,G_2}$ is a
Cohen-Macaulay ideal of dimension $2.$
\end{Corollary}

\section{Properties of Hankel ideals}

\begin{Proposition}\label{2.1}
Let $G_{1}$, $G_{2}$ be connected closed graphs on the vertex set $[m]$, respectively, $[n]$ and let $G$ be the graph associated to the pair $(G_{1},G_{2})$. Then every maximal clique of $G$ is of the form $[a+c-1, b+d-2]$ where $[a,b]$ is a maximal clique of $G_{1}$ and $[c,d]$ is a maximal clique of $G_{2}.$
\end{Proposition}

\begin{proof}
Let $[p,q]$ be a maximal clique of $G$. Then $p={i+k-1}$, $q={j+l-2}$ for some $\{i,j\} \in E(G_{1})$ and $\{k,l\} \in E(G_{2})$. We claim that $[i,j]$ is a maximal clique of $G_{1}$ and $[k,l]$ is a maximal clique of $G_{2}$. We need to prove only the first part of the claim since the second part can be proved in a similar way.

Since $\{i,j\} \in E(G_{1})$  and $G_{1}$ is closed, it follows that $[i,j]$ is a clique of $G_{1}$. Let us assume that $[i,j]$ is not a maximal clique. Then there exists $u \in V(G_{1}), u<i,$ such that $\{u,j\} \in E(G_{1})$ or there exists $v \in V(G_{1}),v>j,$ such that $\{i,v\} \in E(G_{1})$. In the first case, we get $\{u+k-1, j+l-2\} \in E(G)$, which is impossible since ${u+k-1} < {i+k-1}$ and $[i+k-1, j+l-2]$ is a maximal clique of the closed graph $G$. Similarly, if $\{i,v\} \in E(G_{1})$ for some $v>j$, we get $\{i+k-1,v+l-2\} \in E(G)$, again a contradiction by the same argument as above.
\end{proof}

\begin{Remarks}{\em \begin{itemize}
  \item [(1)]  It is clear that if $[a,b]$ is a maximal clique of $G_{1}$ and $[c,d]$ is a maximal clique of $G_{2}$, then $[a+c-1, b+d-2]$ is a clique of $G$. But it might happen that $[a+c-1, b+d-2]$ is not a maximal one.
 For example, let $G_{1}$, $G_{2}$ be closed graphs on the vertex set $[5]$ with the maximal cliques $F_{11}=[1,3],\ F_{12}=[2,4],\ F_{13}=[3,5]$ and $F_{21}=[1,3],\ F_{22}=[2,5],$ respectively. One can easily see that  the maximal cliques $ F_{13}=[3,5]$ and $F_{21}=[1,3]$ give  the clique $[3,6]$ in the associated graph $G$ which is not maximal. Actually, the maximal cliques of $G$ are
$[1,3],[2,6],[3,7],$ and $[4,8]$.

  \item [(2)] The cliques $[a,b]$ and $[c,d]$ in the above proposition are not necessarily  uniquely determined by the maximal clique of $G$.
For example, let $G_{1}$, $G_{2}$ be line graphs on the vertex set $[3]$. The associated graph $G$ is again a line graph on the vertex set $[4]$. Then, the maximal clique $[2,3]$ in $G$ can be obtained either by "adding" the clique $[1,2]$ of $G_1$ with $[2,3]$ of $G_2$ or by using $[2,3]$ from $G_1$ and $[1,2]$ from $G_2.$
\end{itemize}
}
\end{Remarks}

The following theorem collects the main properties of the ideal $I_{G_{1},G_{2}}$. In the statement we use the well-known notation
$\Ass(I)$ and $\Min(I)$ for the associated prime ideals and, respectively, minimal prime ideals of $I.$

\begin{Theorem}\label{2.3}
Let $G_{1}, G_{2}$ be connected closed graphs on the vertex sets $[m]$, respectively [n]. Then:
\begin{itemize}
  \item [(1)] $I_{G_{1},G_{2}}$ is a prime ideal if and only if $G_{1}$ and $G_{2}$ are complete graphs.
  \item [(2)] If at least one of the graphs $G_{1}$, $G_{2}$ is not complete, then $$ \Ass(I_{G_{1},G_{2}})= \Min(I_{G_{1},G_{2}}) = \{ I_{{K_{m},K_{n}}},(x_2,\ldots,x_{m+n-2})\}.$$
  \item [(3)] $I_{G_{1},G_{2}}$ is a set-theoretical complete intersection.
  \item [(4)] $I_{G_{1},G_{2}}$ is a radical ideal if and only one of the following holds:
	\begin{itemize}
		\item [(a)] $G_{1}=K_m$ and either $G_2=K_n$ or
	$\Delta(G_2)=\langle [1,n-1],[2,n]\rangle$;
	\item [(b)] $G_{2}=K_n$ and either $G_1=K_m$ or
	$\Delta(G_1)=\langle [1,m-1],[2,m]\rangle$;
	\end{itemize}
  \item [(5)] The following statements are equivalent:
  \begin{itemize}
    \item  [(a)] $I_{G_{1},G_{2}}$  has a linear resolution;
    \item  [(b)] All powers of $I_{G_{1},G_{2}}$ have a linear resolution;
    \item  [(c)] $I_{G_{1}}$ and $I_{G_{2}}$ have a linear resolution;
    \item  [(d)] $G_{1}$ and $G_{2}$ are complete graphs.
  \end{itemize}
\end{itemize}
\end{Theorem}

\begin{proof}
By Theorem \ref{1.1},  we know that $I_{G_1,G_2} = I_G$ where $G$ is the associated graph of the pair $G_{1}$, $G_{2}$. Hence, in all the statements, we may replace $I_{G_1,G_2}$ by $I_G$.

(1) If $G_{1}=K_{m}$ and $G_{2}=K_{n}$, then $G=K_{m+n-2}$, and the claim is known. Conversely, let $I_{G}$ be a prime ideal. Then, by \cite[Theorem 2.2]{CDE}, it follows that $G$ is a complete graph. Hence $G$ is the clique $[1,m+n-2]$. By Proposition \ref{2.1}, it follows that there exist $[a,b]$ maximal clique in $G_{1}$ and $[c,d]$ maximal clique in $G_{2}$ such that $[a+c-1,b+d-2]=[1,m+n-2]$. This equality implies that $G_{1}=K_{m}$ and $G_{2}=K_{n}$.

(2) follows by (1) and \cite[Theorem 2.2]{CDE}.

(3) This is direct consequence of \cite[Corollary 2.4]{CDE}.

(4) Let us assume that  $G_{2}=K_{n}$ and the facets of the clique complex of $G_{1}$ are $[1,m-1]$ and $[2,m]$. Then one easily sees that the facets of the clique complex of $G$ are $[1,m+n-3]$ and $[2,m+n-2]$. Hence, by using  \cite[Proposition 2.3]{CDE}, it follows that $I_{G}$ is a radical ideal. Let now $I_{G}$ be a radical ideal which is not prime. By \cite[Proposition 2.3]{CDE} it follows that $G$ has two maximal cliques, namely $[1,m+n-3]$ and $[2,m+n-2]$. Let $[a,b]$ and $[c,d]$ be maximal cliques in $G_{1}$, respectively $G_{2}$, such that $[a+c-1,b+d-2]= [1,m+n-3]$. This equality implies that $[a,b]=[1,m-1]$ and $[c,d]=[1,n]$ or $[a,b]=[1,m]$ and $[c,d]=[1,n-1]$. Hence, $G_{1}$ or $G_{2}$ is a complete graph. Let us choose, for instance, $G_{2}=K_{n}$, and assume that $G_{1} \neq K_{m}$. By the form of the cliques of $G$, it follows that $G_{1}$ has the maximal cliques $[1,m-1]$ and $[2,m]$.

The equivalence of the statements in (6) follows by applying \cite[Proposition 2.6]{CDE} and statement (1) in this theorem.
\end{proof}

 In \cite[Theorem 2.7]{CDE} it was shown that, for any closed graph $H$ on the vertex set $[n]$, the regularity of $K[x_1,\ldots,x_{n+1}]/I_{H}$ is at most the number of maximal cliques of $H$. Therefore, we get $\reg(K[x_1,\ldots,x_{n+1}]/I_{H})\leq {n-1}$. If equality holds in this inequality, it follows that $H$ must be the line graph on $[n]$. Conversely, if $H$ is the line graph, then $I_{H}$ is a complete intersection, hence the Koszul complex gives the minimal graded free resolution of $K[x_1,\ldots,x_{n+1}]/I_{H}$. This implies that
$\reg(K[x_1,\ldots,x_{n+1}]/I_{H})= n-1$.

 In our context we get the following result.
\begin{Proposition}\label{2.4}
Let $G_{1}$, $G_{2}$ be connected closed graphs on the vertex set $[m]$, respectively, $[n]$. Then $\reg(S/I_{G_{1},G_{2}}) \leq {m+n-2}$ and the equality holds if and only if $G_{1}$ and $G_{2}$ are line graphs.
\end{Proposition}

\begin{proof}
The inequality follows by Theorem~\ref{1.1}. If $G_{1}$ and $G_{2}$ are line graphs, one may easily check that the associated graph $G$ is a line graph too, hence $\reg(S/I_{G})= m+n-2$. Let us now assume that $\reg(S/I_{G})= m+n-2$. Thus, $I_{G}$ is the line graph on $[m+n-2]$, hence its maximal cliques are $[i,i+1]$ for $1 \leq i \leq m+n-3$. Let us assume, for example, that $G_{1}$ is not a line graph. Therefore, $G_{1}$ has at least one maximal clique $[a,b]$ with $b > a+1$. Then, for any maximal clique $[c,d]$ of $G_{2}$, $[a+c-1,b+d-2]$ is a clique of $G$. But $b+d-2 > (a+c-1)+1$, hence $G$ cannot be a line graph.
\end{proof}
{}

\end{document}